\documentclass{article}  

\usepackage{setspace}
\usepackage{float}
\usepackage{amsmath,amsfonts,amssymb}
\usepackage{amsthm}
\usepackage{graphicx}
\usepackage{epstopdf}
\usepackage{caption}
\usepackage{url}
\usepackage[dvips,margin=2.6cm]{geometry}

\theoremstyle{plain}
\newtheorem{theorem}{Theorem}
\newtheorem{lemma}{Lemma}
\newtheorem{proposition}{Proposition}
\newtheorem{corollary}{Corollary}
\theoremstyle{definition}
\newtheorem{definition}{Definition}
\newtheorem{example}{Example}
\newtheorem{remark}{Remark}

\begin{document}	

\title{Generalized Taylor's formula for power fractional derivatives\thanks{This 
is a preprint of a paper whose final and definite form is published in 
'Bolet\'{\i}n de la Sociedad Matem\'{a}tica Mexicana' at [https://www.springer.com/journal/40590]. 
This work was supported by CIDMA and is funded by the 
Funda\c{c}\~{a}o para a Ci\^{e}ncia e a Tecnologia, I.P. 
(FCT, Funder ID 50110000187) 
under grants UIDB/04106/2020 and UIDP/04106/2020.}}

\author{Hanaa Zitane$^{1}$\\
\url{https://orcid.org/0000-0002-7635-9963}\\	
\texttt{h.zitane@ua.pt}
\and Delfim F. M. Torres$^{1,2,}$\thanks{Corresponding author: \texttt{delfim@ua.pt}}\\
\url{https://orcid.org/0000-0001-8641-2505}\\
\texttt{delfim@ua.pt}; \texttt{delfim@unicv.cv}}

\date{$^{1}$Center for Research and Development in Mathematics and Applications (CIDMA),
Department of Mathematics, University of Aveiro, 3810-193 Aveiro, Portugal\\[0.3cm]
$^{2}$Research Center in Exact Sciences (CICE), 
Faculty of Sciences and Technology (FCT), 
University of Cape Verde (Uni-CV), 
7943-010 Praia, Cape Verde}

\maketitle


\begin{abstract}
We establish a new generalized Taylor's formula for power fractional derivatives 
with nonsingular and nonlocal kernels, which includes many known Taylor's 
formulas in the literature. Moreover, as a consequence, we obtain a general 
version of the classical mean value theorem. We apply our main result to 
approximate functions in Taylor's expansions at a given point. The explicit 
interpolation error is also obtained. The new results are illustrated 
through examples and numerical simulations.

\medskip


\medskip

\end{abstract}


\section{Introduction}

Taylor's theorem is one of the central elementary tools in mathematical analysis, 
e.g., in numerical methods, topology optimization and optimal control~\cite{Blaszczyk,Heydari,Laib}.
It provides simple arithmetic formulas, in polynomial terms, to accurately compute values of various 
transcendental functions, such as trigonometric and exponential ones. This fundamental theorem 
has various significant applications in Mathematics~\cite{Teso}, engineering~\cite{He}, 
and other fields of applied sciences~\cite{ Rani,Shiraishi}. In consequence, 
several versions of Taylor's theorem, involving  different types of fractional operators, 
have been established. For instance, generalizations of this formula, using fractional operators 
with singular kernels, such as Riemann--Liouville and Caputo  derivatives, are provided in 
\cite{Benjemaa,ElAjou,OdibatTaylor,TrujilloTaylor}. In \cite{FernandezTaylor}, a Taylor's theorem 
is proved for Atangana--Baleanu fractional derivatives in Caputo sense while in \cite{ZineTaylor} 
it is derived for generalized weighted fractional operators.

Recently, a new generalized fractional derivative with nonsingular and nonlocal kernel 
was introduced~\cite{PowerDerivative}. The  power fractional derivative (PFD) is essentially 
characterized by the presence of a key power parameter $p$, which allows one to choose the 
appropriate fractional operator that effectively describes the phenomena under study in a natural way, 
and then creating good mathematical models to represent systems and predict their future dynamical behaviors. 
Furthermore, this fractional operator generalizes and unifies most of fractional derivatives 
with nonsingular kernels, such us the Caputo--Fabrizio~\cite{CapFab}, Atangana--Baleanu~\cite{AtanBal}, 
weighted Atangana--Baleanu~\cite{ALRefai1}, and weighed generalized fractional derivatives~\cite{Hattaf}. 

Motivated by available results, in the present paper we propose to investigate a more 
general and rich version of Taylor's formula involving the recently introduced power 
fractional derivative.

The outline of the paper is as follows. In Section~\ref{Sec2}, we review the necessary 
notions on power fractional calculus. Our main results are given in Section~\ref{Sec3}, 
where we begin by proving important lemmas and tools about power fractional operators, 
and their nth-order operators, that are necessary in the sequel. Furthermore, we establish 
a new generalized Taylor's theorem and a general version of the mean value theorem 
via power fractional differentiation. Then, in Section~\ref{Sec4}, we apply our main 
result to approximate functions in Taylor's series at a given point, where the explicit 
interpolation error of the approximation of the function by its Taylor polynomial 
is also characterized. We end up with Section~\ref{Sec5} of conclusion and future work.


\section{Preliminary definitions}
\label{Sec2}

Let $C([a, b])$ be the Banach space of all continuous real functions defined on $[a, b]$, 
where $a, b \in \mathbb{R}$, and $H^{1}(a,b)$ be the Sobolev space of order one defined by 
$$
H^{1}(a,b)=\{f \in L^{2}(a,b): f'\in L^{2}(a,b)\}.
$$
In what follows, we review some basic concepts and tools 
about power fractional calculus that are used along the text. 

\begin{definition}[See \cite{PowerDerivative}]
\label{PowerMittag} 
The power Mittag--Leffler function is given by
\begin{equation}\label{PMF}
{}^p \! E_{k,l}(\tau)=\displaystyle\sum_{n=0}^{+\infty}\dfrac{(\tau\ln p)^{n}}{\Gamma(kn+l)},
\quad \tau \in\mathbb{C},
\end{equation}
where $\min(k, l)>0$, $p>0$, and $\Gamma(\cdot)$ 
is the Gamma function \cite{mittag}.
\end{definition}

\begin{remark} 
Note that the Mittag--Leffler function of two parameters $k$ and $l$ 
is recovered when $p=e$, while the Mittag--Leffler function of one parameter 
$k$ is obtained when $p=e$ and $l=1$ \cite{mittag}. 
\end{remark}

Throughout the paper, we adopt the notations
$$
\chi(\alpha):=\dfrac{1-\alpha}{N(\alpha)}, 
\quad \varphi(\alpha):=\dfrac{\alpha}{N(\alpha)}
\ \text{ and }\ \mu_{\alpha}:=\dfrac{\alpha}{1-\alpha},
$$ 
where $\alpha \in [0, 1)$ and $N(\alpha)$ is a normalization function 
such as $N(0)=N(1^{-})=1$ with 
$N(1^{-})=\underset{\alpha \rightarrow 1^{-}}{\lim}N(\alpha)$.

\begin{definition}[See \cite{PowerDerivative}]
\label{PowerDerivative}
Let $\alpha \in [0, 1)$, $\min(\beta, p)>0$, and $f\in H^{1}(a,b)$. 
The power fractional derivative of order $\alpha$ in the Caputo sense, 
of a function $f$ with respect to the weight function $\omega$, is defined by
\begin{equation}
\label{PFD}
{}^p{}^C \!D_{a,t,\omega}^{\alpha,\beta,p}f(t)=\dfrac{1}{\chi(\alpha)}\dfrac{1}{\omega(t)}
\int_{a}^{t} {}^p \! E_{\beta,1}\left(-\mu_{\alpha}(t-\tau)^{\beta}\right)(\omega f)'(\tau)\, \mathrm{d}\tau,
\end{equation}
where $\omega \in C^{1}([a,b])$ with $\omega>0$ on $[a,b]$.
\end{definition}

\begin{remark}
The PFD \eqref{PFD} generalizes and includes various cases of fractional 
derivative operators available in the literature, such as:
\begin{itemize}
\item  when $p=e$, $\beta=1$, and $\omega(t)\equiv 1$, 
we obtain the Caputo-Fabrizio fractional derivative \cite{CapFab} defined by 
$$
{}^p{}^C \!D_{a,t,1}^{\alpha,1,e}f(t)=\dfrac{1}{\chi(\alpha)}
\int_{a}^{t} \exp\left(-\mu_{\alpha}(t-\tau )\right) f'(\tau)\, \mathrm{d} \tau;
$$

\item when $p=e$, $\beta=\alpha$, and $\omega(t)\equiv 1$, we retrieve  
the Atangana--Baleanu fractional derivative \cite{AtanBal} given by
$$
{}^p{}^C \!D_{a,t,1}^{\alpha,\alpha,e}f(t)=\dfrac{1}{\chi(\alpha)}
\int_{a}^{t} E_{\alpha}\left(-\mu_{\alpha}(t-\tau )\right) f'(\tau)\, \mathrm{d} \tau;
$$

\item when $p=e$ and $\beta=\alpha$, we obtain the weighted 
Atangana--Baleanu fractional derivative given \cite{ALRefai1} by 
$$
{}^p{}^C \!D_{a,t,\omega}^{\alpha,\alpha,e}f(t)
=\dfrac{1}{\chi(\alpha)}\dfrac{1}{\omega(t)}
\int_{a}^{t} E_{\alpha}\left(-\mu_{\alpha}(t
- \tau )^{\alpha}\right)(\omega f)'(\tau)\, \mathrm{d} \tau;
$$

\item when $p=e$, we get the weighted generalized 
fractional derivative \cite{Hattaf} defined as follows:
$$
{}^p{}^C \!D_{a,t,\omega}^{\alpha,\beta,e}f(t)
=\dfrac{1}{\chi(\alpha)}\dfrac{1}{\omega(t)}
\int_{a}^{t} E_{\beta}\left(-\mu_{\alpha}(t
- \tau )^{\beta}\right)(\omega f)'(\tau)\, \mathrm{d} \tau.
$$
\end{itemize}
\end{remark}

Now we provide the power fractional integral (PFI) associated 
with the power fractional derivative \eqref{PFD}.

\begin{definition}[See \cite{PowerDerivative}]	
\label{PowerDIntegral}	
The power fractional integral of order $\alpha$, of a function $f$ 
with respect to the weight function $\omega$, is given by
\begin{equation}
\label{PFI}
{}^p\!I_{a,t,\omega}^{\alpha,\beta,p}f(t)=
\chi(\alpha)f(t)+\ln p\cdot\varphi(\alpha) {}^R\!{}^L\!I_{a,\omega}^{\beta}f(t),
\end{equation}
where ${}^R\!{}^L\!I_{a,\omega}^{\beta}$ is the standard weighted 
Riemann--Liouville fractional integral of order $\beta$ defined by
$$
{}^R\!{}^L\!I_{a,\omega}^{\beta}f(t)=\dfrac{1}{\Gamma(\beta)}\dfrac{1}{\omega(t)}
\int_{a}^{t}(t-\tau )^{\beta-1}(\omega f)(\tau)\, \mathrm{d} \tau .
$$	
\end{definition}

\begin{remark} 
If we let $p=e$ in \eqref{PFI}, then we retrieve the generalized 
fractional integral operator given in \cite{Hattaf}. Moreover, 
if we let $p=e$, $\beta=\alpha$, and $\omega(t)\equiv 1$ in \eqref{PFI}, 
then we obtain the Atangana--Baleanu fractional integral 
operator introduced in \cite{AtanBal}. 
\end{remark}

For the sake of simplicity, we shall denote
${}^p{}^C \!D_{a,t,\omega}^{\alpha,\beta,p}$ 
and ${}^p\!I_{a,t,\omega}^{\alpha,\beta,p}$ 
by ${}^p\!D_{a,\omega}^{\alpha,\beta}$ and ${}^p\!I_{a,\omega}^{\alpha,\beta}$, respectively.


\section{Taylor's formula via power fractional derivatives}
\label{Sec3}

In this section, we establish a new generalized  Taylor's formula in the framework 
of the power fractional derivative. We first prove some fundamental lemmas and 
results about power fractional operators and their nth-order 
that are needed in the proof of the main theorem.

\begin{lemma}
\label{key1}
The power fractional derivative ${}^p\!D_{a,\omega}^{\alpha,\beta}$ 
can be expressed as follows:
\begin{equation}
\label{new}
{}^p\!D_{a,\omega}^{\alpha,\beta}f(t)=\dfrac{1}{\chi(\alpha)}
\displaystyle\sum_{n=0}^{+\infty}\left(-\mu_{\alpha}\ln p\right)^{n} 
{}^R\!{}^L\!I_{a,\omega}^{\beta n+1}\left(\dfrac{(\omega f)'}{\omega}\right)(t).
\end{equation}
This series converges locally and uniformly in $t$ for any $a$, $\alpha$, $\beta$, $p$, 
$\omega$ and $f$, verifying the conditions laid out in Definition~\ref{PowerDerivative}.
\end{lemma}

\begin{proof}
The power Mittag--Leffler function ${}^p \! E_{k,l}(s)$ is an entire function of $s$. 
Since it is locally uniformly convergent in the whole complex plane, 
it implies that the PFD may be rewritten as follows:
\begin{equation*}	
\begin{split}
{}^p\!D_{a,\omega}^{\alpha,\beta}f(t)
&=\dfrac{1}{\chi(\alpha)}\dfrac{1}{\omega(t)}\displaystyle\sum_{n=0}^{+\infty}
\dfrac{\left(-\mu_{\alpha}\ln p\right)^{n}}{\Gamma(\beta n+1)}
\int_{a}^{t}(t-\tau )^{\beta n}(\omega f)'( \tau )\, \mathrm{d}\tau \\
&=\dfrac{1}{\chi(\alpha)}\displaystyle\sum_{n=0}^{+\infty}\left(-\mu_{\alpha}\ln p\right)^{n}
\dfrac{1}{\Gamma(\beta n+1)}\dfrac{1}{\omega(t)}\int_{a}^{t}(t-\tau )^{\beta n}(\omega f)'( \tau )\, \mathrm{d}\tau\\
&=\dfrac{1}{\chi(\alpha)}\displaystyle\sum_{n=0}^{+\infty}
\left(-\mu_{\alpha}\ln p\right)^{n} {}^R\!{}^L\!I_{a,\omega}^{\beta n+1}\left(\dfrac{(\omega f)'}{\omega}\right)(t),
\end{split}	
\end{equation*}	
as required.
\end{proof}

\begin{remark} 
The new formula \eqref{new} of the power fractional derivative \eqref{PFD} 
is easier to handle for certain computational purposes.
\end{remark}

The power fractional integral and derivative satisfy the following composition property.

\begin{proposition}
\label{PFDandPFI} 
Let $\alpha \in [0, 1)$, $p,\beta > 0$ and $f \in H^{1}(a,b)$. Then, it holds that
\begin{equation}
{}^p\!I_{a,\omega}^{\alpha,\beta}\left({}^p\!D_{a,\omega}^{\alpha,\beta}f\right)(t)
=f(t)-\dfrac{(\omega f)(a)}{\omega(t)}.
\end{equation}
\end{proposition}

\begin{proof}
According with Definition~\ref{PowerDIntegral}, we have
$$
{}^p\!I_{a,\omega}^{\alpha,\beta}\left({}^p\!D_{a,\omega}^{\alpha,\beta}f\right)(t)
=\chi(\alpha){}^p\!D_{a,\omega}^{\alpha,\beta}f(t)+\ln p
\cdot\varphi(\alpha) {}^R\!{}^L\!I_{a,\omega}^{\beta}\left({}^p\!D_{a,\omega}^{\alpha,\beta}f\right)(t).
$$
Moreover, by virtue of Lemma~\ref{key1}, one obtains that
\begin{equation*}
\begin{split}
{}^p\!I_{a,\omega}^{\alpha,\beta}\left({}^p\!D_{a,\omega}^{\alpha,\beta}f\right)(t)
&=\displaystyle\sum_{n=0}^{+\infty}\left(-\mu_{\alpha}
\ln p\right)^{n} {}^R\!{}^L\!I_{a,\omega}^{\beta n+1}\left(\dfrac{(\omega f)'}{\omega}\right)(t)\\
&\quad+\mu_{\alpha}\ln p {}^R\!{}^L\!I_{a,\omega}^{\beta}\left[
\displaystyle\sum_{n=0}^{+\infty}\left(-\mu_{\alpha}
\ln p\right)^{n} {}^R\!{}^L\!I_{a,\omega}^{\beta n+1}\left(\dfrac{(\omega f)'}{\omega}\right)(t)\right]\\
&=\displaystyle\sum_{n=0}^{+\infty}\left(-\mu_{\alpha}\ln p\right)^{n} 
{}^R\!{}^L\!I_{a,\omega}^{\beta n+1}\left(\dfrac{(\omega f)'}{\omega}\right)(t)
- \displaystyle\sum_{n=0}^{+\infty}\left(-\mu_{\alpha}
\ln p\right)^{n+1} {}^R\!{}^L\!I_{a,\omega}^{\beta (n+1)+1}\left(\dfrac{(\omega f)'}{\omega}\right)(t)\\
&=\displaystyle\sum_{n=0}^{+\infty}\left(-\mu_{\alpha}\ln p\right)^{n} {}^R\!{}^L
\!I_{a,\omega}^{\beta n+1}\left(\dfrac{(\omega f)'}{\omega}\right)(t)
- \displaystyle\sum_{n=1}^{+\infty}\left(-\mu_{\alpha}\ln p\right)^{n} 
{}^R\!{}^L\!I_{a,\omega}^{\beta n+1}\left(\dfrac{(\omega f)'}{\omega}\right)(t)\\
&={}^R\!{}^L\!I_{a,\omega}^{1}\left(\dfrac{(\omega f)'}{\omega}\right)(t)\\
&=\dfrac{1}{\omega(t)}\int_{a}^{t}(\omega f)'(\tau)\, \mathrm{d}\tau\\
&=f(t)-\dfrac{(\omega f)(a)}{\omega(t)},
\end{split}	
\end{equation*}	
which completes the proof.
\end{proof}

Next result provides the nth-order power fractional integral formula.

\begin{lemma}
\label{PowerInt} 
Let $n\in\mathbb{N}$ and $f\in C([a,b])$. Then, 
\begin{equation}
\label{formula1}
{}^p\!I_{a,\omega}^{n[\alpha,\beta]}f(t)=\displaystyle\sum_{m=0}^{n}C^{m}_{n}
\chi(\alpha)^{n-m}\left(\ln p\cdot\varphi(\alpha)\right)^{m}
\left( {}^R\!{}^L\!I_{a,\omega}^{m\beta} f(t)\right),
\end{equation}	
where $\alpha \in [0, 1)$, $p,\beta > 0$, $t\in [a,b]$ and 
${}^p\!I_{a,\omega}^{n[\alpha,\beta]}={}^p\!I_{a,\omega}^{[\alpha,\beta]}\cdots	{}^p\!I_{a,\omega}^{[\alpha,\beta]}$, 
$n$-times.
\end{lemma}

\begin{proof} 
For $n=0$, formula \eqref{formula1} holds. Indeed, 
one has ${}^p\!I_{a,\omega}^{0\times[\alpha,\beta]}f(t)=f(t)$ and 	
$$
\displaystyle\sum_{m=0}^{0}C^{m}_{0}\chi(\alpha)^{0-m}\left(\ln p
\cdot\varphi(\alpha)\right)^{m}\left( {}^R\!{}^L\!I_{a,\omega}^{m\beta} 
f(t)\right)={}^R\!{}^L\!I_{a,\omega}^{0\times\beta} f(t)=f(t).
$$	
Now, one assumes that formula \eqref{formula1} is satisfied and we prove that 
\begin{equation*}
{}^p\!I_{a,\omega}^{(n+1)[\alpha,\beta]}f(t)
=\displaystyle\sum_{m=0}^{n+1}C^{m}_{n+1}\chi(\alpha)^{n+1-m}
\left(\ln p\cdot\varphi(\alpha)\right)^{m}\left( {}^R\!{}^L\!I_{a,\omega}^{m\beta} f(t)\right)
\end{equation*}
holds true. Indeed, from Definition~\ref{PowerDIntegral}, we have 
\begin{equation*}
\begin{split}
{}^p\!I_{a,\omega}^{(n+1)[\alpha,\beta]}f(t)
&= \chi(\alpha)\left({}^p\!I_{a,\omega}^{n[\alpha,\beta]}f(t)\right)
+\ln p\cdot\varphi(\alpha) {}^R\!{}^L\!I_{a,\omega}^{\beta}\left({}^p\!I_{a,\omega}^{n[\alpha,\beta]}f(t)  \right)\\
&=\chi(\alpha)\left(\displaystyle\sum_{m=0}^{n}C^{m}_{n}\chi(\alpha)^{n-m}
\left(\ln p\cdot\varphi(\alpha)\right)^{m}\left( {}^R\!{}^L\!I_{a,\omega}^{m\beta} f(t)\right)\right)\\
&\quad+\ln p\cdot\varphi(\alpha) {}^R\!{}^L\!I_{a,\omega}^{\beta}\left(\displaystyle
\sum_{m=0}^{n}C^{m}_{n}\chi(\alpha)^{n-m}\left(\ln p
\cdot\varphi(\alpha)\right)^{m}\left( {}^R\!{}^L\!I_{a,\omega}^{m\beta} f(t)\right)  \right)\\
&=\displaystyle\sum_{m=0}^{n}C^{m}_{n}\chi(\alpha)^{n+1-m}\left(
\ln p\cdot\varphi(\alpha)\right)^{m}\left( {}^R\!{}^L\!I_{a,\omega}^{m\beta} f(t)\right)
+\displaystyle\sum_{m=0}^{n}C^{m}_{n}\chi(\alpha)^{n-m}\left(\ln p
\cdot\varphi(\alpha)\right)^{m+1}\left( {}^R\!{}^L\!I_{a,\omega}^{m\beta} f(t)\right)\\
&=\chi(\alpha)^{n+1}f(t)+\displaystyle\sum_{m=1}^{n}C^{m}_{n}\chi(\alpha)^{n+1-m}
\left(\ln p\cdot\varphi(\alpha)\right)^{m}\left( {}^R\!{}^L\!I_{a,\omega}^{m\beta} f(t)\right)\\
&\quad+\displaystyle\sum_{m=1}^{n}C^{m-1}_{n}\chi(\alpha)^{n+1-m}
\left(\ln p\cdot\varphi(\alpha)\right)^{m}\left( {}^R\!{}^L\!I_{a,\omega}^{m\beta} 
f(t)\right)+\left(\ln p\cdot\varphi(\alpha)\right)^{n+1}\left( {}^R\!{}^L\!I_{a,\omega}^{(n+1)\beta} f(t)\right).
\end{split}	
\end{equation*}	
Using the fact that $C^{m}_{n+1}=C^{m}_{n}+C^{m-1}_{n}$, it follows that 
$$
{}^p\!I_{a,\omega}^{(n+1)[\alpha,\beta]}f(t)=\displaystyle\sum_{m=0}^{n+1}C^{m}_{n+1}
\chi(\alpha)^{n+1-m}\left(\ln p\cdot\varphi(\alpha)\right)^{m}
\left( {}^R\!{}^L\!I_{a,\omega}^{m\beta} f(t)\right),
$$
which completes the proof.
\end{proof}

The following result allows us to easily construct our generalized Taylor's formula via power fractional derivatives.

\begin{theorem} 
Assume that  ${}^p\!D_{a,\omega}^{n[\alpha,\beta]}f\in C([a,b])$ 
and ${}^p\!D_{a,\omega}^{(n+1)[\alpha,\beta]}\in C([a,b])$ for 
$\alpha \in [0, 1)$, and $p,\beta > 0$. Then it holds that
\begin{equation}
\label{formula2}
\begin{split}
{}^p\!I_{a,\omega}^{n[\alpha,\beta]}{}^p\!D_{a,\omega}^{n[\alpha,\beta]}f(t)
&-	I_{a,\omega}^{(n+1)[\alpha,\beta]}{}^p\!D_{a,\omega}^{(n+1)[\alpha,\beta]}f(t)\\
&=\dfrac{\omega(a)}{\omega(t)}\left({}^p\!D_{a,\omega}^{n[\alpha,\beta]}f(a)\right)
\displaystyle\sum_{m=0}^{n}C^{m}_{n}\chi(\alpha)^{n-m}\left(\ln 
p\cdot\varphi(\alpha)\right)^{m}\left( 
\dfrac{(t-a)^{m\beta}}{\Gamma(m\beta+1)}\right),
\end{split}	
\end{equation}	
where $t\in [a,b]$ and ${}^p\!D_{a,\omega}^{n[\alpha,\beta]}={}^p\!D_{a,\omega}^{[\alpha,\beta]}
\cdots	{}^p\!D_{a,\omega}^{[\alpha,\beta]}$, $n$-times. 	
\end{theorem}

\begin{proof}
We have 	
\begin{equation*}
\begin{split}
{}^p\!I_{a,\omega}^{n[\alpha,\beta]}{}^p\!D_{a,\omega}^{n[\alpha,\beta]}f(t)
-{}^p\!I_{a,\omega}^{(n+1)[\alpha,\beta]}{}^p\!D_{a,\omega}^{(n+1)[\alpha,\beta]}f(t)
&={}^p\!I_{a,\omega}^{n[\alpha,\beta]}\left({}^p\!D_{a,\omega}^{n[\alpha,\beta]}f(t)
-{}^p\!I_{a,\omega}^{[\alpha,\beta]}{}^p\!D_{a,\omega}^{(n+1)[\alpha,\beta]}f(t)   \right)\\
&={}^p\!I_{a,\omega}^{n[\alpha,\beta]}\left(  {}^p\!D_{a,\omega}^{n[\alpha,\beta]}
f(t)- {}^p\!I_{a,\omega}^{[\alpha,\beta]}{}^p\!D_{a,\omega}^{[\alpha,\beta]}
\left({}^p\!D_{a,\omega}^{n[\alpha,\beta]}f(t)   \right)\right).
\end{split}	
\end{equation*}		
Using Proposition~\ref{PFDandPFI}, one obtains that
\begin{equation*}
\begin{split}
{}^p\!I_{a,\omega}^{n[\alpha,\beta]}{}^p\!D_{a,\omega}^{n[\alpha,\beta]}f(t)
-{}^p\!I_{a,\omega}^{(n+1)[\alpha,\beta]}{}^p\!D_{a,\omega}^{(n+1)[\alpha,\beta]}f(t)
&={}^p\!I_{a,\omega}^{n[\alpha,\beta]}\left(\dfrac{\omega(a){}^p\!D_{a,\omega}^{n[\alpha,\beta]}f(a)}{\omega(t)}\right)\\
&=\omega(a)\left({}^p\!D_{a,\omega}^{n[\alpha,\beta]}f(a)\right){}^p\!I_{a,\omega}^{n[\alpha,\beta]}\left(\dfrac{1}{\omega(t)}\right).
\end{split}	
\end{equation*}	
Then, by virtue of Lemma~\ref{PowerInt}, it follows
\begin{equation*}
\begin{split}
{}^p\!I_{a,\omega}^{n[\alpha,\beta]}{}^p\!D_{a,\omega}^{n[\alpha,\beta]}f(t)
&-	{}^p\!I_{a,\omega}^{(n+1)[\alpha,\beta]}{}^p\!D_{a,\omega}^{(n+1)[\alpha,\beta]}f(t)\\
&=\omega(a)\left({}^p\!D_{a,\omega}^{n[\alpha,\beta]}f(a)\right)\displaystyle
\sum_{m=0}^{n}C^{m}_{n}\chi(\alpha)^{n-m}\left(\ln p\cdot\varphi(\alpha)\right)^{m}
\left( {}^R\!{}^L\!I_{a,\omega}^{m\beta} \left(\dfrac{1}{\omega(t)}\right)\right)\\
&=\dfrac{\omega(a)}{\omega(t)}\left({}^p\!D_{a,\omega}^{n[\alpha,\beta]}f(a)\right)
\displaystyle\sum_{m=0}^{n}C^{m}_{n}\chi(\alpha)^{n-m}
\left(\ln p\cdot\varphi(\alpha)\right)^{m}\left( 
\dfrac{(t-a)^{m\beta}}{\Gamma(m\beta+1)}\right).
\end{split}	
\end{equation*}		
The proof is complete.
\end{proof}

Now, we are able to provide our main result.

\begin{theorem}[Taylor's formula for power fractional derivatives]
\label{TaylorTh}
Assume that ${}^p\!D_{a,\omega}^{m[\alpha,\beta]}\in C([a,b])$ 
for $\alpha \in [0, 1)$, $p,\beta > 0$ and $m=0,1,\cdots,n+1$. 
Then, 
\begin{equation}
\label{Taylor}
\begin{split}
f(t)&=\dfrac{1}{\omega(t)}\left[\omega(a)\displaystyle\sum_{l=0}^{n}{}^p\!D_{a,\omega}^{l[\alpha,\beta]}
f(a)\displaystyle\sum_{m=0}^{l}C^{m}_{l}\chi(\alpha)^{l-m}\left(\ln p\cdot\varphi(\alpha)\right)^{m}
\dfrac{(t-a)^{m\beta}}{\Gamma(m\beta+1)}\right.\\
&\quad+\left. \omega(\lambda) {}^p\!D_{a,\omega}^{(n+1)[\alpha,\beta]}
f(\lambda)\displaystyle\sum_{m=0}^{n+1}C^{m}_{n+1}\chi(\alpha)^{n+1-m}
\left(\ln p\cdot\varphi(\alpha)\right)^{m}\dfrac{(t-a)^{m\beta}}{\Gamma(m\beta+1)} \right],
\end{split}
\end{equation}
where $t\in [a,b]$, $\lambda \in [a,t]$, and ${}^p\!D_{a,\omega}^{l[\alpha,\beta]}
={}^p\!D_{a,\omega}^{[\alpha,\beta]}\cdots	{}^p\!D_{a,\omega}^{[\alpha,\beta]}$, $l$-times.		
\end{theorem}

\begin{proof}
By virtue of formula \eqref{formula2}, one has
\begin{equation*}
\begin{split}
\displaystyle\sum_{l=0}^{n}\left({}^p\!I_{a,\omega}^{l[\alpha,\beta]}{}^p\!D_{a,\omega}^{l[\alpha,\beta]}f(t)\right.
&-\left.	{}^p\!I_{a,\omega}^{(l+1)[\alpha,\beta]}{}^p\!D_{a,\omega}^{(l+1)[\alpha,\beta]}f(t)\right)\\
&=\dfrac{\omega(a)}{\omega(t)}\sum_{l=0}^{n}\left({}^p\!D_{a,\omega}^{l[\alpha,\beta]}f(a)\right)
\displaystyle\sum_{m=0}^{l}C^{m}_{l}\chi(\alpha)^{l-m}\left(\ln p\cdot\varphi(\alpha)\right)^{m}
\left(\dfrac{(t-a)^{m\beta}}{\Gamma(m\beta+1)}\right),
\end{split}	
\end{equation*}
which implies
$$
f(t)-{}^p\!I_{a,\omega}^{(n+1)[\alpha,\beta]}{}^p\!D_{a,\omega}^{(n+1)[\alpha,\beta]}f(t)
=\dfrac{\omega(a)}{\omega(t)}\sum_{l=0}^{n}\left({}^p\!D_{a,\omega}^{n[\alpha,\beta]}
f(a)\right)\displaystyle\sum_{m=0}^{l}C^{m}_{l}\chi(\alpha)^{l-m}
\left(\ln p\cdot\varphi(\alpha)\right)^{m}\left( 
\dfrac{(t-a)^{m\beta}}{\Gamma(m\beta+1)}\right).
$$
Moreover, from Lemma~\ref{PowerInt}, one obtains 
\begin{equation*}
\begin{split}
f(t)&=\dfrac{\omega(a)}{\omega(t)}\sum_{l=0}^{n}\left({}^p\!D_{a,\omega}^{l[\alpha,\beta]}
f(a)\right)\displaystyle\sum_{m=0}^{l}C^{m}_{l}\chi(\alpha)^{l-m}
\left(\ln p\cdot\varphi(\alpha)\right)^{m}\left( 
\dfrac{(t-a)^{m\beta}}{\Gamma(m\beta+1)}\right)\\
&+\displaystyle\sum_{m=0}^{n+1}C^{m}_{n+1}\chi(\alpha)^{n+1-m}
\left(\ln p\cdot\varphi(\alpha)\right)^{m}\left( 
{}^R\!{}^L\!I_{a,\omega}^{m\beta} {}^p\!D_{a,\omega}^{(n+1)[\alpha,\beta]}f(t)\right).
\end{split}	
\end{equation*}
Then, by applying the integral mean value theorem, we deduce that
\begin{equation*}
\begin{split}
f(t)&=\dfrac{1}{\omega(t)}\left[\omega(a)\displaystyle\sum_{l=0}^{n}
{}^p\!D_{a,\omega}^{l[\alpha,\beta]}f(a)\displaystyle\sum_{m=0}^{l}C^{m}_{l}
\chi(\alpha)^{l-m}\left(\ln p\cdot\varphi(\alpha)\right)^{m}
\dfrac{(t-a)^{m\beta}}{\Gamma(m\beta+1)}\right.\\
&\quad+\left. \omega(\lambda) {}^p\!D_{a,\omega}^{(n+1)[\alpha,\beta]}
f(\lambda)\displaystyle\sum_{m=0}^{n+1}C^{m}_{n+1}\chi(\alpha)^{n+1-m}
\left(\ln p\cdot\varphi(\alpha)\right)^{m}\dfrac{(t-a)^{m\beta}}{\Gamma(m\beta+1)} \right],
\end{split}
\end{equation*} 
which completes the proof.
\end{proof}

\begin{remark}
Our Taylor's formula for the power fractional derivative, 
as stated by Theorem~\ref{TaylorTh}, includes most of  
Taylor's formulas without singular kernels
that exist in the literature, such us
\begin{itemize}
\item Taylor's formula involving
the generalized weighted fractional derivative \cite{ZineTaylor}, 
obtained when $p=e$;	
		
\item Taylor's formula involving the weighted Atangana--Baleanu derivative 
in Caputo sense \cite{ALRefai1,ZineTaylor}, obtained when $p=e$ and $\beta=\alpha$;
		
\item Taylor's formula involving the Atangana--Baleanu derivative 
in Caputo sense \cite{FernandezTaylor}, 
obtained when $p=e$, $\beta=\alpha$, and $\omega(t) \equiv 1$.
\end{itemize}
\end{remark}

As a consequence of Theorem~\ref{TaylorTh}, we obtain a
mean value theorem in the framework of power fractional derivative operators.

\begin{corollary}[Mean value theorem for power fractional derivatives]
\label{maeanvalue}
Suppose that $f\in C([a,b])$ and ${}^p\!D_{a,\omega}^{[\alpha,\beta]}f
\in C([a,b])$ for $\alpha \in [0, 1)$ and $p,\beta > 0$. Then, 
\begin{equation}
\label{MVT}
f(t)=\dfrac{1}{\omega(t)}\left[\omega(a)f(a)+ \omega(\lambda) {}^p\!D_{a,\omega}^{[\alpha,\beta]}
f(\lambda)\left(\chi(\alpha)+\ln p\cdot\varphi(\alpha)\dfrac{(t-a)^{\beta}}{\Gamma(\beta+1)} \right) \right],
\end{equation}
where $t\in [a,b]$ and $\lambda \in [a,t]$.		
\end{corollary}

\begin{proof}
The proof follows directly from Theorem~\ref{TaylorTh} by taking $n=0$.
\end{proof}

\begin{remark}
If we let $p=e$, $\omega(t)\equiv 1$ and $\alpha=\beta=1$ 
in Corollary~\ref{maeanvalue}, we obtain the classical mean value theorem. 
\end{remark}


\section{Application: Approximation of functions}
\label{Sec4}

In this section, we apply the developed Taylor's formula for power fractional derivatives \eqref{Taylor} 
to approximate functions at a given point. The approximation method is described in the following result.

\begin{theorem}
\label{approximation}
Suppose that ${}^p\!D_{a,\omega}^{m[\alpha,\beta]}\in C([a,b])$ for $\alpha \in [0, 1)$, 
$p,\beta > 0$ and $m=0,1,\ldots,n+1$. If $t \in [a,b]$, then
\begin{equation}
\label{Apprfunc}
f(t)\simeq {}^p\!A^{\alpha,\beta}_{n}(t)
=\dfrac{\omega(a)}{\omega(t)}\displaystyle\sum_{l=0}^{n}{}^p\!D_{a,\omega}^{l[\alpha,\beta]}
f(a)\displaystyle\sum_{m=0}^{l}C^{m}_{l}\chi(\alpha)^{l-m}\left(\ln p\cdot\varphi(\alpha)\right)^{m}
\dfrac{(t-a)^{m\beta}}{\Gamma(m\beta+1)}.
\end{equation}
In addition, the interpolation error ${}^p\!R^{\alpha,\beta}_{N}$ can be expressed as 
\begin{equation}
\label{error}
{}^p\!R^{\alpha,\beta}_{N}(t):=f(t)-{}^p\!A^{\alpha,\beta}_{N}(t)
=\omega(\lambda){}^p\!D_{a,\omega}^{(N+1)[\alpha,\beta]}f(\lambda)
\displaystyle\sum_{m=0}^{N+1}C^{m}_{N+1}\chi(\alpha)^{N+1-m}\left(\ln p
\cdot\varphi(\alpha)\right)^{m}\dfrac{(t-a)^{m\beta}}{\Gamma(m\beta+1)},
\end{equation}
where $\lambda \in [a,t]$.
\end{theorem}

\begin{proof}
The proof follows directly from Theorem~\ref{TaylorTh}.	
\end{proof}

In order to handle suitably our examples, we first establish the following technical lemma.

\begin{lemma}
\label{technical}
The $l$th-order of the power fractional derivative ${}^p{}^C \!D_{a,t,1}^{l[\alpha,\beta]}$ 
can be expressed as 
\begin{equation*}
{}^p \!D_{a,1}^{l[\alpha,\beta]}f(t)=\dfrac{1}{\chi(\alpha)^{l}}
\displaystyle\sum_{q=0}^{+\infty}C^{l-1}_{q+l-1}\left(
-\mu_{\alpha}\ln p\right)^{q}\left( {}^R\!{}^L\!I_{a,1}^{\beta q+1}f'(t)\right),
\end{equation*}	
where ${}^p\!D_{a,1}^{l[\alpha,\beta]}={}^p\!D_{a,1}^{[\alpha,\beta]}
\cdots	{}^p\!D_{a,1}^{[\alpha,\beta]}$, $l$-times.	
\end{lemma}

\begin{proof}
From Lemma~\ref{key1}, one has
\begin{equation*}
\begin{split}
{}^p \!D_{a,1}^{l[\alpha,\beta]}f(t)
&=\left[\dfrac{1}{\chi(\alpha)}\displaystyle\sum_{n=0}^{+\infty}
\left(-\mu_{\alpha}\ln p\right)^{n} {}^R\!{}^L\!I_{a,\omega}^{\beta n+1}\dfrac{d}{dt}\right]^{l}f(t)\\
&=\dfrac{1}{\chi(\alpha)^{l}}\displaystyle\sum_{n_{1},\ldots,n_{l}}
\left(-\mu_{\alpha}\ln p\right)^{\sum n_{k}} {}^R\!{}^L\!I_{a,\omega}^{\beta \sum n_{k}+1}\dfrac{d}{dt}f(t)\\
&=\dfrac{1}{\chi(\alpha)^{l}}\displaystyle\sum_{q=0}^{+\infty}C^{l-1}_{q+l-1}
\left(-\mu_{\alpha}\ln p\right)^{q}\left( {}^R\!{}^L\!I_{a,1}^{\beta q+1}f'(t)\right),
\end{split}	
\end{equation*}
which completes the proof.	
\end{proof}

Now, we approximate some basic functions about the point $a=0$
using the generalized Taylor series 
given by Theorem~\ref{approximation}.

\begin{example} 
\label{ex01}
Consider the exponential function 
$$
f(t)=\exp(\delta t),\quad\delta>0.
$$
For $\omega(t)\equiv 1$, the weighted fractional integral ${}^R\!{}^L\!I_{a,1}^{\beta}$ 
coincides  with the Riemman--Liouville one 
${}^R\!{}^L\!I_{a}^{\beta}$ \cite{der5} with $\beta>0$. 
Therefore, by virtue  of Lemma~\ref{technical}, one has 
\begin{equation}
\label{key}
{}^p \!D_{a,1}^{l[\alpha,\beta]}f(t)
=\dfrac{1}{\chi(\alpha)^{l}}\displaystyle\sum_{q=0}^{+\infty}C^{l-1}_{q+l-1}
\left(-\mu_{\alpha}\ln p\right)^{q}\left( \delta^{-\beta q} \exp(\delta t)\right),
\end{equation}	
for $\alpha \in [0, 1)$ and $\beta,p>0$. Then, by applying Theorem~\ref{approximation}, 
and taking into account \eqref{key}, it follows that
\begin{equation*}
\exp(\delta t)\simeq\displaystyle\sum_{l=0}^{n}\dfrac{1}{\chi(\alpha)^{l}}
\displaystyle\sum_{q=0}^{+\infty}C^{l-1}_{q+l-1}\left(-\mu_{\alpha}
\ln p\right)^{q}\left(\delta^{-\beta q}\right)\displaystyle
\sum_{m=0}^{l}C^{m}_{l}\chi(\alpha)^{l-m}\left(\ln p\cdot\varphi(\alpha)\right)^{m}
\dfrac{t^{m\beta}}{\Gamma(m\beta+1)},
\end{equation*} 	
where the remainder has the form
\begin{equation*}
{}^p\!R^{\alpha,\beta}_{N}(t)=\dfrac{1}{\chi(\alpha)^{N+1}}\displaystyle
\sum_{q=0}^{+\infty}C^{N}_{q+N}\left(-\mu_{\alpha}\ln p\right)^{q}
\left( \delta^{-\beta q}\exp(\delta \lambda)\right)\displaystyle
\sum_{m=0}^{N+1}C^{m}_{N+1}\chi(\alpha)^{N+1-m}
\left(\ln p\cdot\varphi(\alpha)\right)^{m}\dfrac{t^{m\beta}}{\Gamma(m\beta+1)}
\end{equation*}		
with $\lambda \in [0,t]$.
\end{example}

In contrast with Example~\ref{ex01}, we now consider a cosine function
instead of an exponential.

\begin{example} 
Consider the function 
$$
f(t)=\cos(\delta t), \quad \delta>0.
$$
From Lemma~\ref{technical}, one has
\begin{equation}
\label{key2}
{}^p \!D_{a,1}^{l[\alpha,\beta]}f(t)=\dfrac{1}{\chi(\alpha)^{l}}
\displaystyle\sum_{q=0}^{+\infty}C^{l-1}_{q+l-1}\left(-\mu_{\alpha}
\ln p\right)^{q}\left( \delta^{-\beta q}\cos(\delta t-\beta q \dfrac{\pi}{2})\right),
\end{equation}
for $\alpha \in [0, 1)$ and $\beta,p>0$.  Then, using \eqref{Apprfunc}, 
\eqref{error}, and \eqref{key2}, the function $\cos(\delta t)$ 
can be approximated by
\begin{equation*}
\cos(\delta t)\simeq\displaystyle\sum_{l=0}^{n}\dfrac{1}{\chi(\alpha)^{l}}
\displaystyle\sum_{q=0}^{+\infty}C^{l-1}_{q+l-1}\left(-\mu_{\alpha}
\ln p\right)^{q}\left(\delta^{-\beta q}\cos(\beta q\dfrac{\pi}{2})\right)
\displaystyle\sum_{m=0}^{l}C^{m}_{l}\chi(\alpha)^{l-m}\left(\ln p\cdot\varphi(\alpha)\right)^{m}
\dfrac{t^{m\beta}}{\Gamma(m\beta+1)},
\end{equation*} 	
where the interpolation error is given by
\begin{equation*}
{}^p\!R^{\alpha,\beta}_{N}(t)=\dfrac{1}{\chi(\alpha)^{N+1}}
\displaystyle\sum_{q=0}^{+\infty}C^{N}_{q+N}\left(-\mu_{\alpha}
\ln p\right)^{q}\left( \delta^{-\beta q}\cos(\delta \lambda
-\beta q\dfrac{\pi}{2})\right)\displaystyle\sum_{m=0}^{N+1}C^{m}_{N+1}
\chi(\alpha)^{N+1-m}\left(\ln p\cdot\varphi(\alpha)\right)^{m}
\dfrac{t^{m\beta}}{\Gamma(m\beta+1)}
\end{equation*}			
with $\lambda \in [0,t]$.	
\end{example}

Finally, we expand the sine function as a power Taylor series.

\begin{example}
Consider the function
$$
f(t)=\sin(\delta t),\quad \delta>0.
$$
From Lemma~\ref{technical}, we have
\begin{equation}
\label{key3}
{}^p \!D_{a,1}^{l[\alpha,\beta]}f(t)=\dfrac{1}{\chi(\alpha)^{l}}
\displaystyle\sum_{q=0}^{+\infty}C^{l-1}_{q+l-1}\left(-\mu_{\alpha}
\ln p\right)^{q}\left( \delta^{-\beta q}\sin(\delta t-\beta q \dfrac{\pi}{2})\right),
\end{equation}
for $\alpha \in [0, 1)$ and $\beta,p>0$. Then, using Theorem~\ref{approximation}, 
the approximation of the function $\sin(\delta t)$, in the neighborhood of $a=0$, is given by
\begin{equation}
\label{interp}
\sin(\delta t)\simeq\displaystyle\sum_{l=0}^{n}\dfrac{1}{\chi(\alpha)^{l}}
\displaystyle\sum_{q=0}^{+\infty}C^{l-1}_{q+l-1}\left(-\mu_{\alpha}\ln p\right)^{q}
\left(-\delta^{-\beta q}\sin(\beta q\dfrac{\pi}{2})\right)\displaystyle
\sum_{m=0}^{l}C^{m}_{l}\chi(\alpha)^{l-m}\left(\ln p\cdot\varphi(\alpha)\right)^{m}
\dfrac{t^{m\beta}}{\Gamma(m\beta+1)}
\end{equation} 	
with the remainder
\begin{equation*}
{}^p\!R^{\alpha,\beta}_{N}(t)
=\dfrac{1}{\chi(\alpha)^{N+1}}\displaystyle\sum_{q=0}^{+\infty}C^{N}_{q+N}
\left(-\mu_{\alpha}\ln p\right)^{q}\left( \delta^{-\beta q}
\sin(\delta \lambda-\beta q\dfrac{\pi}{2})\right)\displaystyle
\sum_{m=0}^{N+1}C^{m}_{N+1}\chi(\alpha)^{N+1-m}\left(
\ln p\cdot\varphi(\alpha)\right)^{m}\dfrac{t^{m\beta}}{\Gamma(m\beta+1)},
\end{equation*}	
where $\lambda \in [0,t]$.

We plot the function $f(t)=\sin(\delta t)$ and its Taylor polynomials ${}^p\!A^{\alpha,\beta}_{n}(t)$ 
given by \eqref{interp} around the origin $t=0$, for $\delta=1$, $\alpha=0.1$, $\beta=1.5$ 
and different values of the order $n$, over the interval~ $[0,1]$. We also consider different 
values of the power parameter $p$ to see its effect on the function approximation. 
The results are summarized in Figure~\ref{Figure1}.
\begin{figure}[H]
\begin{center}	
\begin{minipage}{0.47\linewidth}
\centering
\includegraphics[scale=0.57]{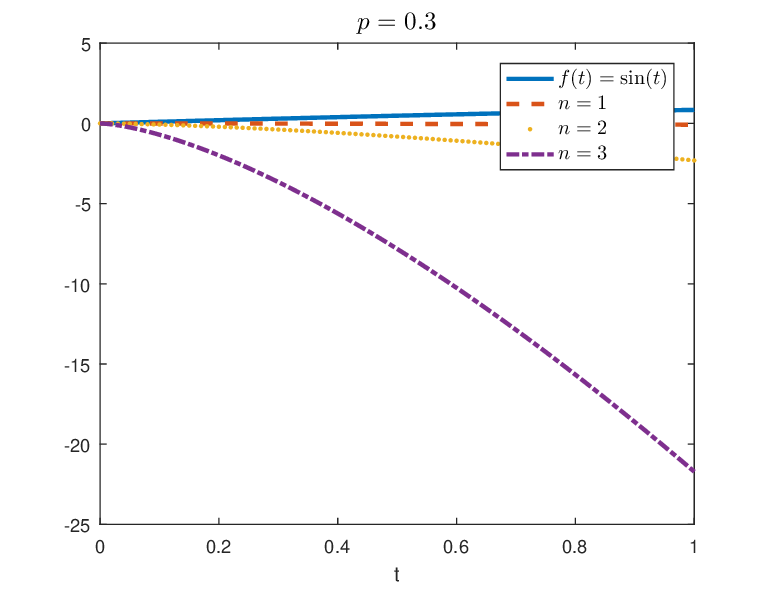}
\end{minipage}
\begin{minipage}{0.47\linewidth}
\centering
\includegraphics[scale=0.57]{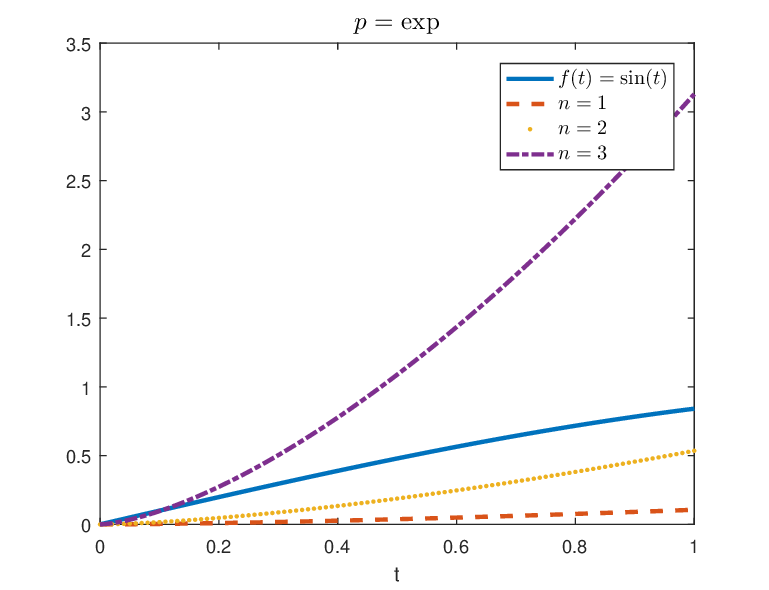}
\end{minipage}
\begin{minipage}{0.47\linewidth}
\centering
\includegraphics[scale=0.57]{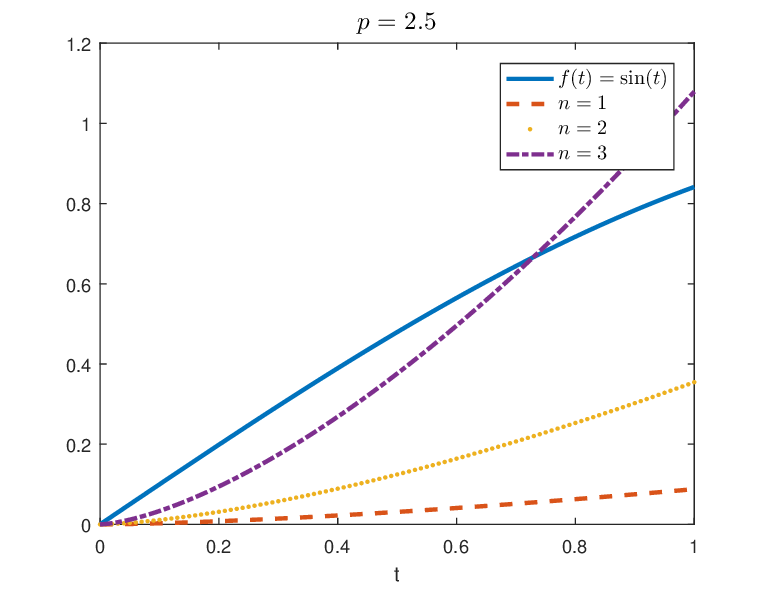}
\end{minipage}
\begin{minipage}{0.47\linewidth}
\centering
\includegraphics[scale=0.57]{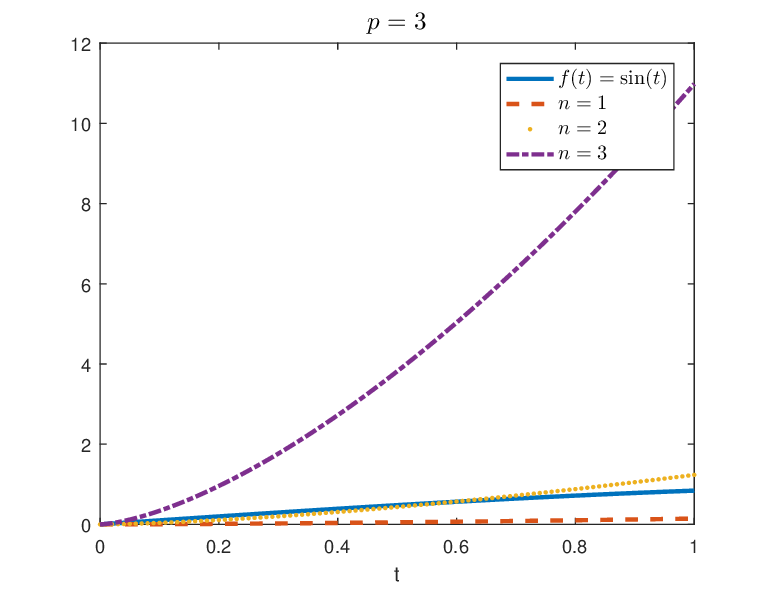}
\end{minipage}
\caption{The function $f(t)=\sin(t)$ and the corresponding Taylor polynomials 
${}^p\!A^{\alpha,\beta}_{n}(t)$ of order $n=1,2,3$, centered at $t = 0$, 
for $\alpha=0.1$, $\beta=1.5$ and different values of $p$.} 
\label{Figure1}
\end{center}
\end{figure}	
\end{example}


\section{Conclusion}
\label{Sec5}

We proved a new generalized Taylor's theorem for power fractional derivatives,  
which extends those available in the literature with nonsingular kernels. 
The proof is based on the establishment of new formulas for nth-order 
of the power fractional operators. We also obtained a general version 
of the mean value theorem. As an application, we used our main result 
to approximate functions at a given point, where the interpolation error 
of the approximation is also given explicitly. Examples and simulations are presented.

Our developed results can be used to solve fractional differential 
equations with non constant coefficients in series form. This will
be addressed elsewhere.


\section*{Statements and Declarations}

\subsection*{Competing and Conflict of Interests} 

The authors have no conflicts or competing of interests to declare.


\subsection*{Data Availability} 

No data associated to the manuscript.


\subsection*{Author contributions}

The conception and design of the study, 
the material preparation, the manuscript writing 
and analysis were performed by Hanaa Zitane 
and Delfim F. M. Torres. Both authors commented 
on all versions of the manuscript and
read and approved the final form of it.



\end{document}